\documentclass{article}
\usepackage{fullpage}
\usepackage{enumitem}
\usepackage{tikz}
\RequirePackage{amsthm,amsmath,amssymb,amsfonts,graphicx, subfigure} 
\RequirePackage[colorlinks]{hyperref}

\renewcommand{\Bbb}{\mathbb}

	\newcommand{\R}{{\mathbb R}}
	\newcommand{\N}{{\mathbb N}}
	\newcommand{\Z}{{\mathbb Z}}
	\newtheorem*{defi}{Definition}
	\newtheorem{lem}{Lemma}
	\newtheorem{obs}{Observation}
	\newtheorem{prop}{Proposition}
	
	\newtheorem{conj}{Conjecture}
	\newtheorem{theo}{Theorem}
	\theoremstyle{remark}

\def\be{\begin{eqnarray}}
\def\ee{\end{eqnarray}}
\def\ben{\begin{eqnarray*}}
\def\een{\end{eqnarray*}}

\newcommand{\setS}{S}
\newcommand{\setT}{T}
\newcommand{\collS}{\mathcal{G}}

\newcommand{\sumS}{\sum_{\setS\in\collS}}

\newcommand{\as}{\alpha_{\setS}}
\newcommand{\bs}{\beta_{\setS}}

\newcommand{\As}{A_{\setS}}

\newcommand{\dth}{\frac{1}{d}}
\newcommand{\calK}{\mathcal{K}}

\newcommand{\vol}{\text{Vol}}
\newcommand{\RL}{{\mathbb R}}
\newcommand{\Rpl}{\R_{+}}
\newcommand{\ra} {\rightarrow}
\newcommand{\conv}{\mathrm{conv}}

\begin{document}
\title{Volumes of subset Minkowski sums and the Lyusternik region}
\author{Franck Barthe\thanks{Institut de Math\'ematiques de Toulouse (UMR 5219). University of Toulouse \& CNRS. UPS, F-31062 Toulouse Cedex 09, France.
E-mail: barthe@math.univ-toulouse.fr} and Mokshay Madiman\thanks{Department of Mathematical Sciences, University of Delaware, 501 Ewing Hall, Newark, DE 19716, USA.
E-mail: madiman@udel.edu}
}
\maketitle

\begin{abstract}
We begin a systematic study of the region of possible values
of the volumes of Minkowski subset sums of a collection of $M$ compact sets in $\RL^d$, which we call the Lyusternik region,
and make some first steps towards describing it. Our main result is that a fractional generalization of the 
Brunn-Minkowski-Lyusternik inequality conjectured by Bobkov et al. (2011) holds in dimension 1. 
Even though  Fradelizi et al. (2016) showed that it fails in general dimension, 
we show that a variant does hold in any dimension. 
\end{abstract}


\section{Introduction}
\label{sec:intro}

The Brunn-Minkowski-Lyusternik inequality is a cornerstone in a number of fields of mathematics-- it appears in geometry as a route to the isoperimetric principle in Euclidean spaces, in algebraic geometry as a route to the Hodge inequality, in functional analysis as a tool in the asymptotic theory of Banach spaces due to the appearance of symmetric convex bodies as their unit balls,
and in probability as the heart of the Pr\'ekopa-Leindler inequality that provides an efficient route to the concentration of measure phenomenon. It states that, for nonempty compact subsets $A, B$ of $\R^d$,
$$
|A+B|^\frac{1}{d} \geq |A|^\frac{1}{d} + |B|^\frac{1}{d} ,
$$
where $|A|$ denotes the volume (Lebesgue measure) of $A$.
First developed for convex sets by Brunn and Minkowski, it was extended by Lyusternik \cite{Lus35} to compact sets, and more generally to Borel sets. The survey \cite{Gar02} is an excellent introduction to the Brunn-Minkowski-Lyusternik inequality, its history, and its many ramifications and connections to other geometric and functional inequalities.

An immediate consequence of the Brunn-Minkowski-Lyusternik inequality is its extension to $M$ sets. If $A_1, A_2, \ldots, A_{M}$ are compact sets in $\R^d$, then
\begin{align}\label{eq:BM-many}
\bigg|\sum_{i=1}^M A_i\bigg|^\frac{1}{d} \geq \sum_{i=1}^M |A_i|^\frac{1}{d} .
\end{align}
Bobkov, Wang, and the second-named author \cite{BMW11} conjectured that this superadditivity property of the functional $|\cdot|^{1/d}$ may be improved.

\begin{conj}\label{conj-bmw}\cite{BMW11}
The functional $A\mapsto |A|^{1/d}$ is fractionally superadditive with respect to Minkowski summation on the class of compact sets in $\RL^d$.
\end{conj}

The fractional superadditivity property is defined in Section~\ref{sec:main}; it would have been a strict improvement of the inequality \eqref{eq:BM-many} for $M$ sets when $M>2$. It was observed in \cite{BMW11} that Conjecture~\ref{conj-bmw} holds for convex sets; thus the interest lay in extending this to general compact, and thence, Borel sets.

Conjecture~\ref{conj-bmw} was motivated by analogies between the inequalities explored in this paper to Information Theory.
The formal resemblance between inequalities in Information Theory and Convex Geometry
was first noticed by Costa and Cover \cite{CC84} but has since been extensively developed.
For example, there now exist entropy analogues of the Blaschke-Santal\'o inequality \cite{LYZ04},
the reverse Brunn-Minkowski inequality \cite{BM11:cras, BM12:jfa}, the Rogers-Shephard inequality \cite{BM13, MK18} and the Busemann inequality \cite{BNT16}.
Indeed, volume inequalities, entropy inequalities, and certain small ball inequalities can be unified using
the framework of R\'enyi entropies \cite{WM14, MMX17:1}; the surveys \cite{DCT91, MMX17:0} may be 
consulted for much more in this vein. 
On the other hand, natural analogues in the Brunn-Minkowski theory of inequalities from Information Theory
hold sometimes but not always  \cite{FGM03, AFO14, FM14}.
Another related set of results has to do with Schur-concavity of entropy or volume in various settings;
see \cite{MNT20} for details.

Let $X$ be a random vector taking values in $\RL^d$,
with density function $f_X$ 
(with respect to Lebesgue measure $dx$).
Define the entropy of $X$ by $h(X) = -\int f_X(x) \log f_X(x) dx$ if the integral exists and $- \infty$ otherwise (see, e.g., \cite{CT91:book}). The entropy power of $X$ is  $N(X)=\exp\{2h(X)/d\}$.
The functional $A\mapsto \vol_d(A)^{1/d}$ in the geometry of compact subsets of $\R^d$, 
and the functional $f_X\mapsto N(X)$ in probability are analogues in the resemblance discussed above.
The superadditivity property $N(X+Y)\geq N(X)+N(Y)$ for independent random vectors,
which is called the Shannon-Stam entropy power inequality \cite{Sha48, Sta59} is then the analogue of
the Brunn-Minkowski-Lyusternik inequality. Fractional superadditivity of the entropy power was established
in stages: by \cite{ABBN04:1} for the leave-one-out case in a paper that was celebrated because it
resolved a conjecture regarding the central limit theorem (simpler proofs were given by \cite{MB06:isit, TV06, Shl07}), 
for a larger class of hypergraphs by \cite{MB07}, and finally in full by \cite{MG19}, where the Stam region
(which is like the Lyusternik region that we define and explore in this paper, but for entropy powers) was defined and explored.
Conjecture~\ref{conj-bmw} is the precise analogue in this dictionary
of the  fractional superadditivity of entropy power established by \cite{MG19}.

Therefore, it was rather surprising when \cite{FMMZ16} constructed a counterexample to establish that Conjecture~\ref{conj-bmw} fails in dimension 12 and above; soon after, \cite{FLZ22} found a counterexample in dimension 7. This provides another example where the analogy between Euclidean geometry and Information Theory breaks down. The goal of this note is to show that, in fact, the fractional superadditivity conjecture of \cite{BMW11} {\it does} hold in dimension 1. Moreover, a variant of Conjecture~\ref{conj-bmw} does hold in general dimension-- namely, the volume functional itself (without an exponent) is fractionally superadditive with respect to Minkowski summation on the class of compact sets in $\RL^d$.

This note is organized as follows. In Section~\ref{sec:main}, we describe our main results carefully, giving all necessary definitions along the way  and also some of the shorter proofs.  Sections~\ref{sec:fsa-pf} and \ref{sec:other-pf} are devoted to proving the theorems described in Section~\ref{sec:main}-- specifically,
Section~\ref{sec:fsa-pf} proves in stages
the fractional superadditivity statement in Theorem~\ref{thm:FSA}, which is the technically
most demanding part of this note, while
Section~\ref{sec:other-pf} contains the proofs
of the subsequent theorems in Section~\ref{sec:main}.
We supplement this main part of the paper with some discussion and open questions in Section~\ref{sec:rmks}, and with some reasons why we believe fractional superadditivity is an important structural property of set functions and therefore worthy of study in the Appendix (Section~\ref{sec:app}).

\section{Main Results}
\label{sec:main}

Let  ${\bf \calK}_{d}$ be the collection of nonempty compact sets in $\R^d$.
We write $[M]$ for the index set $\{1,2,\ldots,M\}$,
and $\emptyset$ for the empty set.  
For any nonempty $\setS\subset [M]$, 
and any $A_1, A_2, \ldots, A_{M}\in {\bf \calK}_{d}$,
define the Minkowski subset sum
\ben
\As=\sum_{i\in\setS} A_{i} .
\een
We are interested in the volumes of the subset sums $\As$ (denoted $|\As|$), 
which leads naturally to the following objects of study.

\begin{defi}\label{defi:stam}
Let ${\bf \calK}_{d}^M$ be the collection of all $M$-tuples  ${\bf A}=(A_1, \ldots, A_M)$
of nonempty compact subsets $A_i$ of $\RL^d$.
Define the set function $\nu_{\bf A}:2^{[M]} \ra\Rpl:=[0,\infty)$ by $\nu_{\bf A}(\emptyset)=0$ and
\be\label{nu}
\nu_{\bf A}(\setS)\,=| \As | \,= \bigg|\sum_{i\in\setS} A_i\bigg| 
\ee 
for nonempty $\setS\subset [M]$.
The $(d,M)$-Lyusternik region is 
\ben
\Lambda_d(M)=\{ \nu_{\bf A} : {\bf A}\in {\bf \calK}_{d}^M \}.
\een
By arranging the elements of $2^{[M]}=\{S_1, S_2, \ldots, S_{2^M}\}$
according to the shortlex order\footnote{We merely need to fix any total order on $2^{[M]}$; we choose the shortlex order, which first orders the sets by cardinality and then lexicographically within sets of a given cardinality, for convenience. Thus $S_1=\emptyset, S_2=\{1\},S_3=\{2\}, \ldots, S_{M+1}~=\{M\}, S_{M+2}=\{1,2\}, \ldots, S_{2^M}=[M]\}$. Whenever we write $(a_S:S\subset [M])$, what we mean is the $2^M$-tuple $(a_{S_1}, a_{S_2}, \ldots,a_{S_{2^M}})$, where the coordinates are indexed in the shortlex order.},
we may identify the set $\Lambda_d(M)$ of set functions
with the subset of $(\Rpl)^{2^{M}}$ given by
$\{(|A_{S_1}|, |A_{S_2}|, \ldots, |A_{S_{2^M}}|): {\bf A}\in {\bf \calK}_{d}^M \}$.
\end{defi}

We name these regions  (we will use the description as a collection of set functions or as a collection of points in $(\Rpl)^{2^{M}}$ interchangeably, but this will always be obvious by context) after L. A. Lyusternik in honor of his pioneering role 
\cite{Lus35} in the study of volumes of Minkowski sums, especially when dealing with
sets that are not necessarily convex.
Clearly, any inequality that relates volumes of different subset sums gives a bound on the Lyusternik region.
Conversely, knowing the Lyusternik region is equivalent, in principle, to knowing all volume inequalities that hold
for Minkowski sums of general collections of compact sets, and all that do not.

Let $(\collS,\beta)$ be a {\it weighted hypergraph} on a set $T$, i.e.,
a collection $\collS$ of subsets of $T$ (which we may think of as ``hyperedges''),
together with a weight function $\beta:\collS \to \Rpl$ that assigns weight $\beta_S=\beta(S)$ to each
set $\setS$ in $\collS$.
We say that  $(\collS,\beta)$ is a {\it fractional partition} of $T$ if for each $i\in T$, we have
$\sum_{S\in \collS:\, i\in S} \beta_S=1$.
These conditions can be phrased as a single one, using the characteristic functions $\mathbf1_S:T\to \{0,1\}$,
as
$\sum_{S\in \collS} \beta_S \mathbf1_S={\mathbf 1}_{T}=1.$

We say that a set function $v:2^{[M]}\ra\Rpl$ is {\it fractionally superadditive} if for every subset $T\subset [M]$
\be\label{fsa}
v(T) \geq \sumS \bs v(\setS)
\ee
holds for every fractional partition $(\collS,\beta)$ of $T$.
Write $\Gamma_{FSA}(M)$ for the class of all  fractionally superadditive 
set functions  $v$ with $v(\emptyset)=0$. 

The set function $v:2^{[M]}\ra\Rpl$ is said to be {\it supermodular} if 
\ben
v(\setS\cup\setT) + v(\setS\cap\setT)  \geq v(\setS) + v(\setT)  
\een
for all sets $\setS,\setT\subset [M]$. Write $\Gamma_{SM}(M)$ for the class of all supermodular
set functions $v$ with $v(\emptyset)=0$.

It is known \cite{MP82, MT10} that for $M\geq 3$, $\Gamma_{SM}(M)\subsetneq \Gamma_{FSA}(M)$, i.e.,
every supermodular set function is fractionally superadditive but not vice versa. For $M=2$, given the limited
availability of subsets, it is easy to see that  $\Gamma_{SM}(2)= \Gamma_{FSA}(2)$, and both are equal
to the class of superadditive set functions.

We start with two straightforward observations that set the stage for further discussion.

\begin{obs}\label{obs:cone}
For each $d, M \in\N^*$, $\Lambda_d(M)$ is a cone, 
which is invariant under the natural action of the symmetric group on $M$ elements.
\end{obs}

\begin{proof}
To prove the first part, suppose 
$\nu_{{\bf A}}\in \Lambda_d(M)$, with ${\bf A}=(A_1, \ldots, A_M)$ and compact sets $A_i\subset\R^d$. For any $\lambda>0$, consider ${\bf A}'=(\lambda^\dth A_1, \ldots, \lambda^\dth A_M)$.
Clearly $\nu_{{\bf A}'}=\lambda \nu_{{\bf A}} $, hence $\lambda \nu_{{\bf A}}$ belongs to $\Lambda_d(M)$.
This is also true for $\lambda=0$, since $0\in \Lambda_{d}(M)$ (it can be realized e.g. with singletons). Hence $\Lambda_d(M)$ is a cone.



Let us address the second part. If ${\bf A}=(A_1, \ldots, A_M)$ realizes $(\as: \setS\in [M])\in \Lambda_{d}(M)$,
then clearly $\pi({\bf A}):=(A_{\pi(1)}, \ldots, A_{\pi(M)})$ realizes the vector 
$$
(\alpha_{\pi(\setS)}: \setS\in [M]),
$$ 
where $\pi(\setS):=\{\pi(i):i\in\setS\}$. Thus 
$
(\alpha_{\pi(\setS)}: \setS\in [M])\in \Lambda_{d}(M).
$
\end{proof}

\begin{obs}\label{obs:L2}
For all $d\in \N^* $,
\[\Lambda_d(2)=\left \{(0,a,b,c)\in\Rpl^4: c\geq \big(a^{\frac1d}+b^{\frac1d}\big)^d \right\}.\]
When $d=1$, this can be rephrased as
$\Lambda_1(2)=\Gamma_{SM}(2)=\Gamma_{FSA}(2)$.
\end{obs}

\begin{proof}
The Brunn-Minkowski-Lyusternik inequality states that all nonempty compact subsets  $A_1,A_2$ of $\R^d$ verify $|A_1+A_2|^{\frac1d}\ge |A_1|^{\frac1d}+|A_2|^{\frac1d}$; this proves the inclusion of $\Lambda_d(2)$ in the above set. To see the reverse inclusion,
we need to show that for any triple $(a,b,c)\in\Rpl^3$ with $c^{\frac1d}\geq a^{\frac1d}+b^{\frac1d}$,  there exists a pair of compact sets $A_1, A_2$ in $\R$
with $|A_1|=a, |A_2|=b$ and $|A_1+A_2|=c$. 

We start with the case when $a$ or $b$ is strictly positive.
Without loss of generality (thanks to the invariance property mentioned in Observation~\ref{obs:cone}), we may assume $a>0$, $b\ge 0$ and $c^{\frac1d} \ge a^{\frac1d}+b^{\frac1d}$.
We may find $q\in \N$ and $r\in[0,a^{\frac1d})$ such that
\[\frac{c-\big(a^{\frac1d}+b^{\frac1d}\big)^d}{a^{\frac{d-1}{d}}}=q a^{\frac1d}+r.\]
Let $e_1$ be the first vector in the canonical basis of $\R^d$. Consider $A_1=[0,a^{\frac1d}]^d$ and 
$$A_2=[-b^{\frac1d},0]^d\cup \bigg(\bigcup_{i=1}^{q} \big\{ia^{\frac1d}e_1\big\}\bigg)\cup \big\{\big(qa^{\frac1d} +r\big)e_1\big\},$$
(we understand the union within parentheses to
be empty when $q=0$). Then
\begin{eqnarray*}
A_1+A_2 &=&[-b^{\frac1d},a^{\frac1d}]^d\cup \left( \bigg(\bigcup_{i=1}^q [ia^{\frac1d},(i+1)a^{\frac1d}]\}\bigg)\cup [qa^{\frac1d} +r,(q+1)a^{\frac1d}+r]\right)\times [0,a^{\frac1d}]^{d-1}\\
&=&[-b^{\frac1d},a^{\frac1d}]^d\cup \left(  [a^{\frac1d},(q+1)a^{\frac1d}+r]\times [0,a^{\frac1d}]^{d-1}\right).
\end{eqnarray*}
Consequently $|A_1|=a$, $|A_2|=b$ and $|A_1+A_2|=\big(a^{\frac1d}+b^{\frac1d}\big)^d+(qa^{\frac1d}+r)a^{\frac{d-1}{d}} =c$; thus we are done.

It remains to deal with triples of the form $(0,0,c)$ with $c\ge0$.  By the cone property of Observation \ref{obs:cone} it is enough to deal with one $c>0$. This is very easy in dimension $d\ge 2$, by considering lower dimensional cubes $A_1=[0,1]\times \{0\}^{d-1}$, $A_2= \{0\} \times[0,1]^{d-1}$, which have measure 0 and sum up to the full cube $[0,1]^d$. In dimension 1, we can still use sets of lower dimensions: consider the Cantor ternary set $\mathcal C=\cap_{n\in \mathbb N} E_n$, where $E_0=[0,1]$ and for all $n\ge 0$, $E_{n+1}=\frac{1}{3}E_n\cup \left(\frac23+\frac13 E_n \right)$. It is classical that this compact set has measure zero, and contains all numbers which can be expressed as $\sum_{k\ge 1} x_k 3^{-k}$ for some sequence $(x_k)$ taking values in $\{0,2\}$ (in other words, numbers in $[0,1]$ admitting an expansion in base 3 involving only digits 0 and 2). As a consequence, $\frac12 \mathcal C$ contains numbers in  $[0,1]$ admitting an expansion in base 3 involving only digits 0 and 1, and it is clear that 
\[
[0,1]\subset \mathcal C + \frac12 \mathcal C \subset \left[0, \frac32\right].
\]
We have put forward two sets of measure 0, with a sum of positive measure. This completes the proof of Observation \ref{obs:L2}.
\end{proof}

Observation~\ref{obs:L2} gives a complete description of the Lyusternik region
for the case where one has only two sets. 
This naturally gives rise to the question that is the main focus of this paper: 
what is the relationship between $\Lambda_d(M)$ on the one hand, and $\Gamma_{FSA}(M)$ or $\Gamma_{SM}(M)$ on the other, when $M\geq 3$?
The following statement sums up our contribution to this problem:

\begin{theo}\label{thm:beyondFSA}
For any $d\in\N^*$ and $M\geq 3$,
\begin{itemize}
    \item $\Lambda_d(M) \subsetneq \Gamma_{FSA}(M)$,
    \item $\Lambda_d(M)$ and $\Gamma_{SM}(M)$ have nonempty intersection but neither is a subset of the other.
\end{itemize}
\end{theo}

\begin{proof}
The inclusion $\Lambda_d(M) \subset \Gamma_{FSA}(M)$ comes from the inequality in Theorem \ref{thm:FSA} below. 
The fact that the inequality is strict is a consequence of the second item, since $\Gamma_{SM}(M)\subset \Gamma_{FSA}(M)$.

Next we turn to the proof of the second part of the theorem. By taking all the sets $A_i$ to be singletons, 
it is clear that the zero function (which assigns the value 0 to every subset of $[M]$) is in $\Lambda_d(M)\cap \Gamma_{SM}(M)$; hence this intersection is nonempty. 

It was observed in \cite{FMMZ18} that the volume is not supermodular already in dimension 1.
Indeed, they considered the sets 
 $A_1 = \{0, 1\}$ and $A_2 = A_3 = [0, 1]$. Then,
$|A_1 + A_2 + A_3| + |A_1| = 3 < 4 = |A_1 + A_2| + |A_1 + A_3|$.
Consequently it is clear that 
for any $M\geq 3$, $\Lambda_1(M) \nsubseteq \Gamma_{SM}(M)$. This example can be adapted to cover the case of dimensions $d\ge 2$: let $k\in \N^*$, and consider 
$A_1=[k]^d=\{0,\ldots,k\}^d$, $A_2=A_3=[0,1]^d$. Plainly $A_1+A_2=A_1+A_3=[0,k+1]^d$ and 
$A_1 + A_2 + A_3=[0,k+2]^d$, and for $k$ large enough 
\[|A_1 + A_2 + A_3| + |A_1| = (k+2)^d < 2(k+1)^d= |A_1 + A_2| + |A_1 + A_3|.\]

Finally, we need to show that volumes of partial sums cannot reach all  supermodular set functions. Consider the set function $\alpha:2^{[M]}\to \Rpl$ defined by $\alpha(S)=\mathrm{Card}(S)$, which is clearly supermodular. If $\alpha$ was in $\Lambda_d(M)$, there would be compact sets in $\R^d$ with, in particular, $|A_i|=\mathrm{Card}(\{i\})=1$ and $|A_1+A_2|=\mathrm{Card}(\{1,2\})=2$. In dimension $d\ge 2$ this is impossible since the Brunn-Minkowski inequality ensures that $|A_1+A_2|\ge \big(|A_1|^{\frac1d}+|A_2|^{\frac1d}\big)^d=2^d>2$.

To deal with dimension $d=1$, we consider the set function $\beta$ defined by $\beta([M])=M+1=\alpha([M])+1$ and for $S\subsetneq [M]$, $\beta(S)=\mathrm{Card}(S)=\alpha(S)$. It is still supermodular, since increasing the value of a supermodular function on the full set only improves the supermodularity property. If $\beta$ was in $\Lambda_1(M)$, we would have compact sets in $\R$ with $|A_i|=1$, $|A_i+A_j|=2$ for $i\neq j$. Therefore $A_i, A_j$ are an equality case of the one-dimensional Brunn-Minkowski inequality, which ensures that they are intervals (see, e.g., \cite{HM53} or \cite[Section 8]{BZ88:book}), of length 1. This implies that $|A_1+\cdots+A_M|=M<\beta(M).$ Hence $\beta\not\in \Lambda_1(M)$.
\end{proof}

Our main result is the following fractional superadditivity property:

\begin{theo}\label{thm:FSA}
For any fractional partition  $(\collS,\beta)$ of $[M]$,
\begin{equation}\label{eq:BMF1}
\big|A_1+\cdots+A_M\big|\ge \sumS \beta_S \, \Big|\sum_{i\in S} A_i\Big|
\end{equation}
holds for nonempty compact subsets  $A_1,\ldots,A_M$  of $\R^d$. In dimension $d=1$, the inequality is an equality when all sets $S\subset [M]$ with $\beta_S>0$ satisfy that $\sum_{i\in S} A_i$ is an interval.
\end{theo}

Theorem~\ref{thm:FSA} is proved in stages in Section~\ref{sec:fsa-pf}. First, it is shown in 
Section~\ref{ss:fsa-red} that it suffices to consider so-called ``regular fractional partitions''. After a discussion of some examples to set notation in Section~\ref{ss:fsa-eg},
the proof for the case $d=1$ is detailed in Section~\ref{ss:fsa-real}.
Finally, the extension to general finite dimension is done in
Section~\ref{ss:fsa-dim}, following ideas of \cite{FMMZ18} where a similar approach is used for the leave-one-out fractional partition.

We note that it is possible to have equality in \eqref{eq:BMF1} for non-convex sets-- for example, one can consider $A_1=A_2=A_3=[0,\frac{1}{2}]\cup [1,\frac{3}{2}]$, and the leave-one-out hypergraph (i.e., $\collS=\{\{1,2\}, \{2,3\}, \{3,1\}\}$ is the collection of all sets of cardinality 2, and each $\beta_S=\frac{1}{2}$).

Theorem~\ref{thm:FSA} may be compared with
Conjecture~\ref{conj-bmw}, originally proposed in
\cite[Conjecture 3.1]{BMW11}. This conjecture proposed that
\begin{eqnarray}\label{conjdimn}
\big|A_1+\cdots+A_M\big|^\frac{1}{d}\ge \sumS \beta_S \left| \sum_{i\in\setS} A_i\right|^\frac{1}{d},
\end{eqnarray}
for all compact sets, and \cite[Theorem 3.7]{BMW11} verified the same for convex sets. The motivation of \cite{BMW11} came from the fact that the conjectured inequality would have provided a fundamental refinement of the Brunn-Minkowski inequality for 3 or more sets.
However, even a very special case of the inequality \eqref{conjdimn} (involving 
a particular fractional partition and all sets $A_i$ being copies of the same compact set $A$) was shown by \cite{FMMZ16} to fail in dimension 12 and above, and by \cite{FLZ22} to fail in dimension 7 and above.
Nonetheless, Theorem~\ref{thm:FSA} shows in particular that the conjectured inequality \eqref{conjdimn} is 
true for all compact sets in dimension 1.
Moreover, Theorem~\ref{thm:FSA} also shows that 
a bound similar to \eqref{conjdimn} continues to hold for arbitrary compact sets in general dimension, but at the cost of removing the exponent $1/d$ on the volume.


In fact, we also have a positive result in general dimension for a special class of sets. 

\begin{theo}\label{th:prod}
Let $(\collS,\beta)$ be a fractional partition of $[M]$. Fix $0\leq k\leq M$. For each $i\in [k]$, suppose $C_{i,1}$, 
$\ldots, C_{i,M}$ are nonempty compact convex subsets of $\R^{d_{i}}$.
For each $k+1\leq i\leq L$, suppose $C_{i,1}, \ldots, C_{i,M}$  are nonempty compact subsets of $\R$.
Let $A_j=C_{1,j}\times \ldots \times C_{L,j}$, so that each $A_i$ is a compact subset of $\R^d$, with $d=L-k +\sum_{i=1}^k d_i$. Then 
$$\big|A_1+\cdots+A_M\big|^\frac{1}{d} \ge \sum_{S\in \collS} \beta_S \, \Big|\sum_{i\in S} A_i\Big|^\frac{1}{d}.$$
\end{theo}

\begin{proof}
Combining  Proposition~\ref{prop:cartes} (which we will prove in Section~\ref{ss:cartes})
with our main result (fractional superadditivity for $d=1$) and the fractional Brunn-Minkowski inequality for convex bodies
observed in \cite{BMW11}, we obtain Theorem~\ref{th:prod}.
\end{proof}

If we consider the special case of Theorem~\ref{th:prod} where $k=0$ (i.e., each $A_j$ is a Cartesian product of one-dimensional compact sets), standard approximation arguments yield that one can extend the statement to Cartesian product of one-dimensional Borel sets. In other words, Theorem~\ref{th:prod} implies that the conjecture of \cite{BMW11} does hold for ``Borel-measurable rectangles with axis-parallel sides''.



It is natural to ask if the phenomena investigated thus far for Minkowski sums in finite-dimensional real vector spaces
also have analogues in a discrete setting, i.e., for Minkowski sums of finite subsets of a discrete group, with volume replaced
by cardinality. One would expect such discrete analogues to be relevant to the field of additive combinatorics, as they are related
to the Cauchy-Davenport inequality.  We observe that an analogue does hold in the group of integers, extending a result of 
Gyarmati, Matolcsi and Ruzsa \cite{GMR10}.

\begin{theo}\label{th:Z}
Let $(\collS,\beta)$ be a fractional partition of $[M]$. Let $A_1,\ldots,A_M$ be nonempty finite subsets of $\Z$. Then 
$$\#(A_1+\cdots+A_M)-1\ge \sum_{S\in \collS} \beta_S \, \Big[ \#\left(\sum_{i\in S} A_i\right)-1\Big],$$
where $\#(S)$ denotes the cardinality of $S$ for any finite
$S\subset \mathbb Z$. The inequality is an equality when there exists $\rho\in\N$ such that  all $S\subset [M]$ with $\beta_S>0$ verify that  $\sum_{i\in S} A_i$ is an arithmetic progression of increment $\rho$.
\end{theo}

Theorem~\ref{th:Z} is proved in Section~\ref{ss:integer-pf}. Note that the leave-one-out case of Theorem~\ref{th:Z} was proved by \cite{GMR10}. For other related inequalities, the reader may consult \cite{MMT10:itw, MMT12, WWM14:isit, MWW21, MWW19}.

Finally we remark that while the study in this paper has focused on compact sets, analogous objects are clearly of interest and highly nontrivial to characterize even if we restrict to convex sets. Indeed, characterizing the possible volumes of Minkowski sums of convex sets is closely related to describing the possible collections of mixed volumes, and some comments on the relevant literature are made in Section~\ref{sec:rmks}.

\section{Proof of Theorem~\ref{thm:FSA}}
\label{sec:fsa-pf}

\subsection{A reduction to regular fractional partitions}
\label{ss:fsa-red}

A first step in the proof of Theorem~\ref{thm:FSA} is to reduce to regular fractional partitions. First note that a
fractional partition can be viewed as a map  defined on the power set of $[M]$, $\beta:2^{[M]}\to [0,1]$,
where the collection of subsets of $M$ for which $\beta$ is nonzero (which we call the ``support'' of $\beta$) is the collection $\collS$ in our original definition of fractional partitions in Section~\ref{sec:main}.
The fractional partition condition becomes
$$\sum_{S\subset [M]} \beta_S \mathbf1_S=\mathbf 1.$$
Obviously the term corresponding to $S=\emptyset$ is superfluous,
so that we may represent the set of possible fractional partitions of $[M]$
as follows:
\be\label{eq:FM}
\mathcal F_{M}=\left\{ (\beta_S)_{\emptyset \neq S \subset [M]};\;  \forall S,\, \beta_S\ge 0 \mbox{ and } \forall i\in [M], \sum_{S; \; i\in S} \beta_S=1
\right\}.
\ee
Note that the above conditions ensure that $\beta_S\le 1$.

A {\em regular} fractional partition is a fractional partition that is constant on its support, i.e., $\beta(S)=c$ for $S\in\collS$, and $\beta(S)=0$ otherwise. The defining condition can then be written as
$c\cdot\#\{S\in\collS:i\in S)=1$ for each $i\in [M]$, which means that $c=1/q$ for a positive integer $q$, and each index is contained in exactly $q$ elements of $\collS$. In other words, $\collS$ is a $q$-regular hypergraph as commonly defined in combinatorics, whence the terminology.

The representation \eqref{eq:FM} of $\mathcal F_M$ shows that it is a compact polyhedral convex set. 
Any of its extreme points $\beta$ is the unique point in $\mathcal F_M$ satisfying $\beta_S=0$ for all $S$ in a certain collection 
$\collS^c\subset 2^{[M]}$. This means 
that $(\beta_S)_{S\in \collS}$ is the unique solution of a system of the form: for all $i\in [M]$,
$\sum_{S \in \collS;\; i\in S} \beta_S=1$. 
Hence this system is invertible
and since it has rational coefficients, we get that the nonzero coefficients $\beta_S$
are rational. Hence, we have shown\footnote{This is not claimed to be new; similar arguments and conclusions
appear, e.g., in \cite{Sha67, GG08}, which also contain additional information about extreme
fractional partitions. For example, \cite{GG08} show that one needs to allow denominators of the rational numbers
that appear in extreme fractional partitions to grow at least exponentially in $M$.} that extreme fractional partitions 
(i.e., extreme points of $\mathcal F_M$)
only involve rational  coefficients $\beta_S$. 

In order to prove Inequality \eqref{eq:BMF1} for all fractional partitions (i.e., to show that $\mathcal F_M$ lies in the halfspace defined by the linear inequality \eqref{eq:BMF1}), 
it is enough to prove it for the extreme fractional partitions. 
In particular, it is enough to deal with partitions with 
$\beta_S\in \mathbb Q$ for all $S$.
 Writing these coefficients as fractions with the same denominator $q$
and allowing to repeat sets (as many times as the numerator of their coefficient by $\beta$), 
we can reduce to the following simpler setting:
$S_1,\ldots,S_s$ are subsets of $[M]$ and verify
\begin{equation}
\label{eq:q-cover}
\sum_{j=1}^s \mathbf1_{S_j}=q\mathbf 1,
\end{equation}
or equivalently, for each $i\in [M]$, there are exactly $q$ indices $j$ such that $i\in S_j$.
This means  that $[M]$ is covered exactly $q$ times by the sets $(S_j)_{1\le j\le s}$.
Observe that because of repetitions, we use a finite sequence of sets, rather than a collection
of sets.
Under the above assumption \eqref{eq:q-cover}, our task is to show that
$$q\big|A_1+\cdots+A_M\big|\ge \sum_{j=1}^s  \Big|\sum_{i\in S_j} A_i\Big|.$$

\subsection{Starting with examples}
\label{ss:fsa-eg}

Gyarmati, Matolcsi and Ruzsa \cite{GMR10} have dealt (for subsets of $\mathbb Z$), with the "leave-one-out" case where the fractional 
partition is made of all the subsets of $[M]$ with cardinality $M-1$ and equal weights. Their argument is based on decompositions of  
the small sumsets and a double counting argument. As noted  in \cite{FMMZ16} it also  works for subsets of $\mathbb R$. 
As a warm-up we present the simplest non-trivial 
case of $M=3$, for subsets of $\mathbb R$ and the fractional partition
$$ \mathbf 1_{\{1,2,3\}}=\frac12 \left( \mathbf  1_{\{1,2\}}+ \mathbf 1_{\{2,3\}}+ \mathbf 1_{\{3,1\}}\right).$$
Let $A_1,A_2,A_3$ be three nonempty compact subsets of $\mathbb R$.
Assume that $\min(A_i)=0$ and denote $a_i:=\max(A_i)$.
Since $0$ belongs to all $A_i$'s, the following inclusions hold:
\begin{eqnarray}\label{123}
 (A_1+A_2)\cup \, {}_{a_1+a_2<}(a_1+A_2+A_3)& \subset & A_1+A_2+A_3\\
 (A_2+A_3)_{\le a_2} \cup (A_1+a_2+A_3) &\subset & A_1+A_2+A_3, \nonumber
 \end{eqnarray}
where ${}_{t<}S:=S\cap(t,+\infty)$ and $S_{\le t}:=S\cap(-\infty, t]$. By construction
the unions are essentially disjoint (sets intersect in at most one point), hence passing to lengths
 and summing up the corresponding two inequalities gives
\begin{eqnarray*}
2 |A_1+A_2+A_3|&\ge&  |A_1+A_2|+ |{}_{a_1+a_2<}(a_1+A_2+A_3)|+ |(A_2+A_3)_{\le a_2} |+|A_1+a_2+A_3|\\
&=&  |A_1+A_2|+ |{}_{a_2<}(A_2+A_3)|+ |(A_2+A_3)_{\le a_2} |+|A_1+A_3|\\
&=&  |A_1+A_2|+ |A_2+A_3|+|A_1+A_3| .
\end{eqnarray*}
One can  cook up by hand such decompositions for slightly more complicated fractional partitions. In order to explain our strategy 
for general regular partitions, let us put forward some features of the above decomposition. Since this is only meant to 
explain where our forthcoming formal proof comes from, we do not try to  give formal definitions.

We shall say that an element $i\in [M]$ is covered by a term in the above decompositions (i.e. a truncated sumset),  if this term contains  
a translate of $A_i$ (or rather of $A_i\setminus\{0,a_i\}$). 
This is actually a property of the formula rather than of the sets.
 
For instance $A_1+A_2$ covers $1$ since $A_1\subset A_1+A_2$. 
It also covers $2$, but not 3. 
The term ${}_{a_1+a_2<}(a_1+A_2+A_3)$ covers 3 since it contains $a_1+a_2+A_3$
(more precisely $a_1+a_2+A_3\setminus\{0\})$. It does not cover 1, neither 2.

If we rewrite the decompositions \eqref{123} and underline in each term the indices which it covers, we get: 
\begin{eqnarray}\label{123underline1}
 (\underline{A_1}+\underline{A_2})\cup {}_{a_1+a_2<}(a_1+A_2+\underline{A_3})& \subset & A_1+A_2+A_3\\
 (\underline{A_2}+A_3)_{\le a_2} \cup (\underline{A_1}+a_2+\underline{A_3}) &\subset & A_1+A_2+A_3,\label{123underline2}
 \end{eqnarray}
we observe that each decomposition covers  every index once. We can encode
this on the incidence matrix of the regular partition (columns correspond to elements $i\in [M]$ and lines to the sets in the partition): for each decomposition we connect
the couples $(i,S)$ where $i$ is covered by a term involving a translation of $\sum_{k\in S}A_k$. For the first line \eqref{123underline1}, we connect $(1,\{1,2\})$ to $(2,\{1,2\})$, and then $(2,\{1,2\})$ to $(3,\{2,3\})$. As we observed that all indices are covered this line is a graph of a function on $[M]=\{1,2,3\}$.  For \eqref{123underline2}, we connect $(1,\{1,3\})$ to $(2,\{,2,3\})$, and then $(2,\{,2,3\})$ to $(3,\{1,3\})$. We get Figure \ref{fig3}.  We remark that the connections are only drawn for easy visualization of our procedure in terms of graphs of set-valued functions on $\{1,2,3\}$ drawn as though we would draw a graph of a function from the real line to itself; it is not important, for instance, that we 
did not connect $(1,\{1,3\})$ to $(3,\{1,3\})$.

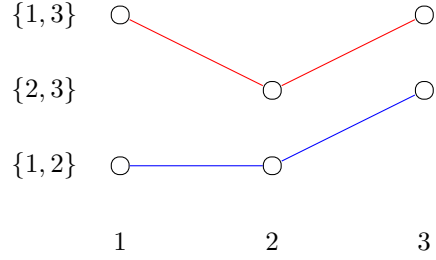
\begin{figure}
\center
\begin{tikzpicture}
	\tikzstyle{indice}=[rectangle]
	\tikzstyle{rond}=[rectangle,draw,rounded corners=3pt]
	\node[indice](1) at (0,0){1};
	\node[indice](2) at (2,0){2};
	\node[indice](3) at (4,0){3};
	\node[indice](12) at (-1,1){$\{1,2\}$};
	\node[indice](23) at (-1,2){$\{2,3\}$};
	\node[indice](13) at (-1,3){$\{1,3\}$};
	\node[rond](1e12) at (0,1){};
	\node[rond](2e12) at (2,1){};
	\node[rond](2e23) at (2,2){};
	\node[rond](3e23) at (4,2){};
	\node[rond](1e13) at (0,3){};
	\node[rond](3e13) at (4,3){};
	\tikzstyle{lienb}=[-,color=blue]
	\tikzstyle{lienr}=[-,color=red]
	\draw[lienb] (1e12) -- (2e12)-- (3e23);
	\draw[lienr] (1e13) -- (2e23) -- (3e13);
\end{tikzpicture}
\caption{Leave-one-out partition on $\{1,2,3\}$}
\label{fig3}
\end{figure}
Reading this simple figure from bottom to top, one can recover the decompositions \eqref{123}: the bottom graph (in blue) corresponds to \eqref{123underline1};
again we read by considering lines (corresponding to sets) from  bottom up: we use 
$A_1+A_2$ to cover 1 and 2, and then $A_2+A_3$ to cover 3, but we truncate it from below at $a_1+a_2$ (corresponding to 1 and 2 being already covered) for disjointness. 
Next we pass to the upper graph, and consider sets starting from below: the first relevant one is $A_2+A_3$ which we use to cover 2 only, so we truncate it from above at $a_2$. 
Next we use $A_1+A_3$, translated by $a_2$ (note that the translation corresponds to the previously covered index).

The main feature of the figure is that it contains the graphs of two functions on $\{1,2,3\}$ which do not cross. Since  the partition is regular, they are uniquely determined by this property.

\medskip
Let us try this reverse engineering approach in a more intricate situation:
$$ 3\mathbf1_{[5]}= \mathbf 1_{\{2,3\}}+\mathbf 1_{\{1,2,4\}}+\mathbf 1_{\{1,2,4,5\}}+\mathbf 1_{\{1,3,5\}}+\mathbf 1_{\{3,4,5\}}.$$
We start with plotting the incidence table of this fractional partition in Figure \ref{fig5} and we draw the corresponding non-crossing graphs. Next we use them
in order to build decompositions of sumsets.

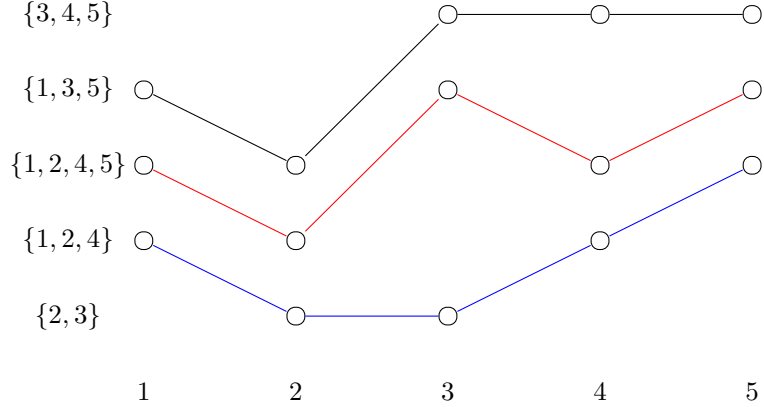
\begin{figure}
\center
\begin{tikzpicture}
	\tikzstyle{indice}=[rectangle]
	\tikzstyle{rond}=[rectangle,draw,rounded corners=3pt]
	\node[indice](1) at (0,0){1};
	\node[indice](2) at (2,0){2};
	\node[indice](3) at (4,0){3};
	\node[indice](4) at (6,0){4};
	\node[indice](5) at (8,0){5};
	\node[indice](23) at (-1,1){$\{2,3\}$};
	\node[indice](124) at (-1,2){$\{1,2,4\}$};
	\node[indice](1245) at (-1,3){$\{1,2,4,5\}$};
	\node[indice](135) at (-1,4){$\{1,3,5\}$};
	\node[indice](345) at (-1,5){$\{3,4,5\}$};
	\node[rond](2e23) at (2,1){};
	\node[rond](3e23) at (4,1){};
	\node[rond](1e124) at (0,2){};
	\node[rond](2e124) at (2,2){};
	\node[rond](4e124) at (6,2){};
	\node[rond](1e1245) at (0,3){};
	\node[rond](2e1245) at (2,3){};
	\node[rond](4e1245) at (6,3){};
	\node[rond](5e1245) at (8,3){};
	\node[rond](1e135) at (0,4){};
	\node[rond](3e135) at (4,4){};
	\node[rond](5e135) at (8,4){};
	\node[rond](3e345) at (4,5){};
	\node[rond](4e345) at (6,5){};
	\node[rond](5e345) at (8,5){};
	\tikzstyle{lienb}=[-,color=blue]
	\tikzstyle{lienr}=[-,color=red]
	\tikzstyle{liend}=[-,color=black]
	\draw[lienb] (1e124) -- (2e23)-- (3e23)-- (4e124) -- (5e1245);
    \draw[lienr] (1e1245) -- (2e124)-- (3e135)--(4e1245)--(5e135);
       \draw[liend] (1e135) -- (2e1245)-- (3e345)--(4e345)--(5e345);
\end{tikzpicture}
\caption{A 3-regular partition of $\{1,2,3,4,5\}$}
\label{fig5}
\end{figure}

We start with the bottom graph (in blue), of a function defined on the set $[5]$ . We consider the sets of the partition starting from the bottom, and at each step we want to cover exactly the indices $1\le i\le 5$ which are on the set and on the graph:
\begin{itemize}
\item The first set is $\{2,3\}$ it is entirely on the blue graph: we need to cover 2 and 3, so can simply take $A_2+A_3$.  We could  write $(\underline{A_2}+\underline{A_3})_{\le a_2+a_3}$ even if the truncation is useless here, in order
to stress  our goal: cover 2 and 3, but nothing more.
\item We move up and consider the next set $\{1,2,4\}$. Our goal is to cover the
indices 1 and 4 (which correspond to the dots on the graph at the height of  $\{1,2,4\}$), using $A_1+A_2+A_4$ translated by as many $a_i$ as we can, for previously covered indices $i$. The truncation from below is imposed by the upper 
bound on the previous set, the one from above by the fact that we do not want to cover more indices than 1 and 4 (it is superfluous in this case). The only choice is 
  $$ {}_{a_2+a_3<}(\underline{A_1}+A_2+a_3+\underline{A_4})_{\le a_1+a_2+a_3+a_4}.$$
  \item The last dot on the blue graph  is at the third line, so we have to use the
  set $\{1,2,4,5\}$ in order to cover the last uncovered index 5. By similar considerations we are led to choose:
    $$ {}_{a_1+a_2+a_3+a_4<} (A_1+A_2+a_3+A_4+\underline{A_5}).$$
 \end{itemize}
Summing up, using the blue line, we have obtained the following disjoint 
union inside the full sum set:
\begin{equation}\label{eq:5-1}
(\underline{A_2}+\underline{A_3}) \bigcup {}_{a_2+a_3<}(\underline{A_1}+A_2+a_3+\underline{A_4})\bigcup
{}_{a_1+a_2+a_3+a_4<} (A_1+A_2+a_3+A_4+\underline{A_5}) \subset \sum_{i=1}^5 A_i.
\end{equation}
Next we deal with the second graph (in red).
\begin{itemize}
\item The first relevant set (starting from bottom) is $\{1,2,4\}$ and only the point $(2,\{1,2,4\})$ is on the red  graph. So we select $(A_1+\underline{A_2}+A_4)_{\le a_2}$
\item Going up one set, the points $(1,\{1,2,4,5\})$ and $(4,\{1,2,4,5\})$ are on the red graph so we need to use $\{1,2,4,5\}$ to cover 1 and 4. This leads to 
$$ {}_{a_2<}( \underline{A_1}+A_2+ \underline{A_4}+A_5)_{\le a_1+a_2+a_4}.$$
\item Eventually considering the set $\{1,3,5\}$ which meets the red graph at 3 and 5 we choose
$$  {}_{a_1+a_2+a_4<}(A_1+a_2+\underline{A_3}+a_4+\underline{A_5}).$$
\end{itemize}
Summing up, the red line leads to the inclusion
\begin{equation}\label{eq:5-2}
(A_1+\underline{A_2}+A_4)_{\le a_2} \bigcup  {}_{a_2<}( \underline{A_1}+A_2+ \underline{A_4}+A_5)_{\le a_1+a_2+a_4} \bigcup  {}_{a_1+a_2+a_4<}(A_1+a_2+\underline{A_3}+a_4+\underline{A_5})
 \subset \sum_{i=1}^5 A_i.
\end{equation}
The same procedure for the upper graph (in black) gives
\begin{equation}\label{eq:5-3}
(A_1+\underline{A_2}+A_4+A_5)_{\le a_2} \bigcup (\underline{A_1}+a_2+A_3+A_5)_{\le a_1+a_2} \bigcup (a_1+a_2+\underline{A_3}+\underline{A_4}+\underline{A_5})
 \subset \sum_{i=1}^5 A_i.
\end{equation}

Eventually, passing to length of sets in the inclusions \eqref{eq:5-1}, \eqref{eq:5-2}, \eqref{eq:5-3}, adding everything up and  collecting the pieces of the various sumsets gives 
$$ 3 \Big| \sum_{i=1}^5 A_i\Big| \ge |A_2+A_3|+ |A_1+A_2+A_4|+ |A_1+A_2+A_4+A_5|+ |A_1+A_3+A_5|+ |A_3+A_4+A_5|.$$
After these examples, we are ready for the general case.

\subsection{Proof for the real line}
\label{ss:fsa-real}

Let us proceed to some simplifications and introduce concise notation. First of all, by translation 
invariance of Lebesgue's measure, we can translate all the sets and assume that for all $i\in [M]$, $\min A_i=0$. Then we denote $a_i:=\max A_i$.
Viewing $A$ as a function from $[M]$ to $2^\R$ and $a$ as a function from $[M]$ to $\R_+$,
we write 
$$ \sum_S A:= \sum_{i\in S} A_i \qquad \mathrm{and} \qquad \sum_S a:= \sum_{i\in S} a_i.$$
With this notation our goal is to show that 
\begin{equation}
\label{eq:qBM}
 q\big|\sum_{[M]}A\big|\ge \sum_{j=1}^s  \Big|\sum_{S_j} A\Big|.
 \end{equation}
Our proof of this inequality will rely on the arbitrary choice of an order of the sets (which was
already made in the notation  $(S_j)_{1\le j\le s}$).

\medskip
By the $q$-covering hypothesis, each $i\in [M]$ belongs to exactly $q$ of the sets $(S_j)_{j=1}^s$.
Hence there are indices
 $$1\le h_1(i)<h_2(i)<\cdots < h_q(i)\le s$$
 such that $i$ belongs to the sets having these indices, and to these sets only:
 $i\in S_{h_k(i)}$ for all $k$ with $1\le k\le q$. Hence, we have built $q$ functions $h_1,\ldots, h_q$ from $[M]$ to $[s]$. They will play a central role in the argument. For each of these functions, 
 we prove a lower bound on the length of the full sum $\sum_{[M]}A$:
 
 \begin{lem}\label{lem:decomp}
 Let $k\in[q]$ be an integer between 1 and $q$. Then
 \begin{equation}\label{eq:U1}
  \bigcup_{j=1}^s \left(\left(\sum_{S_j}A+\sum_{h_k^{-1}([1,j-1])\setminus S_j }a \right)\cap \left(\sum_{h_k^{-1}([1,j-1])}a\; ;\sum_{h_k^{-1}([1,j])}a\right]\right)\subset \sum_{[M]} A,
  \end{equation}
 where the union is disjoint. Hence, passing to lengths of sets:
  \begin{equation}\label{eq:U2}
   \sum_{j=1}^s \left|\left(\sum_{S_j}A+\sum_{h_k^{-1}([1,j-1])\setminus S_j }a \right)\cap \left(\sum_{h_k^{-1}([1,j-1])}a\; ;\sum_{h_k^{-1}([1,j])}a\right] \right|\le  \Big| \sum_{[M]} A\Big|.
 \end{equation}  
 \end{lem}
 Observe that quite a few of the above sets can be empty. For instance when $j=1$,
 $[1,j-1]=\emptyset$ and $\sum_{h_k^{-1}([1,j-1])}a=0$ as a sum on the empty set.
More importantly, when $j$ does not belong to the range of $h_k$ the interval 
 $\left(\sum_{h_k^{-1}([1,j-1])}a\; ;\sum_{h_k^{-1}([1,j])}a\right]$ is also empty.
 
\begin{proof}
Since for all $i\in [M]$, $0\in A_i$ and $a_i\in A_i$ it is plain that
 $$\sum_{S_j}A+\sum_{h_k^{-1}([1,j-1])\setminus S_j }a \subset  \sum_{[M]} A,$$
 hence the inclusion is proved.
 The fact that the union is disjoint comes from the disjointness of the intervals 
 $$\left(\sum_{h_k^{-1}([1,j-1])}a\; ;\sum_{h_k^{-1}([1,j])}a\right].$$
 Indeed, since $a_i\ge 0$, $j_1\le j_2$ implies that  $\sum_{h_k^{-1}([1,j_1])}a \le \sum_{h_k^{-1}([1,j_2])}a$.

\end{proof} 

In order to prove the fractional inequality \eqref{eq:qBM}, we sum up the inequalities provided
by the above lemma, for $k$ ranging from 1 to $q$. Permuting sums, we obtain
\begin{eqnarray*}
q\big|\sum_{[M]}A\big| 
&\ge & \sum_{j=1}^s  \sum_{k=1}^q  \left|\left(\sum_{S_j}A+\sum_{h_k^{-1}([1,j-1])\setminus S_j }a \right)\cap \left(\sum_{h_k^{-1}([1,j-1])}a\; ;\sum_{h_k^{-1}([1,j])}a\right] \right| \\
&=& \sum_{j=1}^s  \left( \sum_{k=1}^q  \left|\sum_{S_j}A \cap \left(\sum_{h_k^{-1}([1,j-1])\cap S_j}a\; ;\sum_{h_k^{-1}([1,j])\setminus (h_k^{-1}([1,j-1])\setminus S_j )}a\right] \right| \right) .
\end{eqnarray*}
Using in the first place that $B\setminus (C\setminus D)= (B\setminus C)\cup (B\cap D)$,
and then the inclusion $h_k^{-1}(\{j\})\subset S_j$ (which follows from the definitions), we get that 
\begin{eqnarray*}
h_k^{-1}([1,j])\setminus (h_k^{-1}([1,j-1])\setminus S_j )&=&h_k^{-1}(\{j\})\cup (h_k^{-1}([1,j])\cap S_j)\\
 &=& h_k^{-1}([1,j])\cap S_j.
 \end{eqnarray*}
 Hence we have shown that 
\begin{equation}
\label{eq:last}
q\big|\sum_{[M]}A\big| \ge 
\sum_{j=1}^s  \left( \sum_{k=1}^q  \left|\sum_{S_j}A \cap \left(\sum_{h_k^{-1}([1,j-1])\cap S_j}a\; ;\sum_{h_k^{-1}([1,j])\cap S_j }a\right] \right| \right) .
\end{equation} 
In order to combine the terms in the inner sum, we need some observations on the 
end-points of the various intervals.

\begin{lem}
Let $1\le j\le s$ and $1\le k\le q-1$ be integers. Then
\begin{enumerate}
\item $\displaystyle h_q^{-1}\big([1,j-1]\big)\cap S_j=\emptyset$
\item $\displaystyle h_1^{-1}\big([1,j]\big)\cap S_j=S_j$
\item $\displaystyle h_{k+1}^{-1}\big([1,j]\big)\cap S_j=  h_k^{-1}\big([1,j-1]\big)\cap S_j$
\end{enumerate}
\end{lem}
\begin{proof}
The first point is obvious when $j=1$ since $[1,j-1]=\emptyset$ in that case. If $j\ge 1$, and 
if $i\in \displaystyle h_q^{-1}\big([1,j-1]\big)\cap S_j$ then $h_q(i)\le j-1$. Therefore, by definition
the $q$ sets to which $i$ belongs have indices $h_1(i)<\cdots<h_q(i)\le j-1$. This contradicts the
fact that $i\in S_j$. 

To prove the second point, it is enough to show that $S_j\subset h_1^{-1}\big([1,j]\big)$. This is also very simple: if $i\in S_j$ then by definition there exists $1\le \ell\le q$ such that $j=h_\ell(i)$.
Therefore $1\le h_1(i)\le h_\ell(i)=j$.

Let us address the third point, by establishing inclusions in both directions. First, assume that $i\in S_j$ and $h_{k+1}(i)\le j$. By definition $h_k(i)<h_{k+1}(i)$. Since these are integer numbers, 
$h_k(i)\le h_{k+1}(i)-1\le j-1$. This proves that  $\displaystyle h_{k+1}^{-1}\big([1,j]\big)\cap S_j\subset  h_k^{-1}\big([1,j-1]\big)\cap S_j$.

Conversely, assume that $i\in S_j$ and $h_k(i)\le j-1$. Since $i$ belongs to $S_j$ there exists $\ell$ such that $h_\ell(i)=j$. It follows that $h_k(i)<h_\ell(i)=j$. Since $t\mapsto h_t(i)$ is strictly 
increasing, we can deduce that $k<\ell$, that is $k+1\le \ell$. Consequently $h_{k+1}(i)\le h_\ell(i)=j$. Thus we have shown that $i\in h_{k+1}^{-1}\big([1,j]\big)\cap S_j$. The proof of the lemma is complete.
\end{proof}
Let us explain how to conclude the proof, resuming at \eqref{eq:last}.
By the latter lemma, 
$$\left(\sum_{h_q^{-1}([1,j-1])\cap S_j}a\; ;\sum_{h_q^{-1}([1,j])\cap S_j }a\right]=
  \left(\sum_{\emptyset }a\; ;\sum_{h_q^{-1}([1,j])\cap S_j }a\right] =
  \left(0\; ;\sum_{h_q^{-1}([1,j])\cap S_j }a\right] ,$$
and for all $k$ such that $1\le k\le q-1$  
$$\left(\sum_{h_k^{-1}([1,j-1])\cap S_j}a\; ;\sum_{h_k^{-1}([1,j])\cap S_j }a\right] 
=\left(\sum_{h_{k+1}^{-1}([1,j])\cap S_j}a\; ;\sum_{h_k^{-1}([1,j])\cap S_j }a\right].$$
Recalling $h_k<h_{k+1}$, it is then clear that the above $q$ intervals are disjoint, and that
their union is 
$$  \left(0\; ;\sum_{h_1^{-1}([1,j])\cap S_j }a\right]= \left(0\; ;\sum_{ S_j }a\right],$$
where we have used the second point of the lemma in the last step. Using this information, 
we may rewrite \eqref{eq:last} as
\begin{equation}
\label{eq:recollect}
q\big|\sum_{[M]}A\big| \ge 
\sum_{j=1}^s  \left|\sum_{S_j}A \cap \Big(0\; ;\sum_{S_j }a\Big] \right|  .
\end{equation}
Recall that $A_i\subset[\min(A_i),\max(A_i)]=[0,a_i]$, hence
$\sum_{S_j}A \subset  \Big[0\; ;\sum_{S_j }a\Big]$ and actually it contains $0$.
Since a point is Lebesgue negligible, we have
$$\left|\sum_{S_j}A \cap \Big(0\; ;\sum_{S_j }a\Big] \right| =\Big|\sum_{S_j}A  \Big|,$$
and the fractional inequality is established.

\medskip
Eventually we check the sufficient condition for equality claimed in Theorem \ref{thm:FSA} for $d=1$. Without loss of generality we may assume that for all $S$, $\beta_S>0$  (otherwise we remove the sets from the fractional partition) and that $\min(A_i)=0$, $\max(A_i)=a_i$. Summing up the inclusions $\{0,a_i\}\subset A_i \subset [0,a_i]$ gives for all $S\subset [M]$,
\[ \left\{0,\sum_S a\right\}\subset \sum_S A \subset \left[0,\sum_S a\right]. \] 
Since by hypothesis $\sum_S A$ is an interval for $S\in \mathcal G$, we get $|\sum_S A|= \sum_{i\in S} a_i$. Hence, using the fractional partition,
\[\sum_{S\in \mathcal G} \beta_S \big|\sum_S A\big|=\sum_{S\in \mathcal G} \beta_S \left(\sum_{i\in S} a_i\right) =\sum_{i=1}^M a_i \left( \sum_{S\in\mathcal G; i\in S}  \beta_S\right)= \sum_{i=1}^M a_i.\]
Moreover the above inclusion implies that $|\sum_{[M]} A|\le \sum_{i=1}^M a_i=\sum_{S\in \mathcal G}\beta_S \big|\sum_S A\big|$, which should be combined to the general inequality $|\sum_{[M]} A|\ge \sum_{S\in \mathcal G} \beta_S\big|\sum_S A\big|$ in order to get equality.

\subsection{Extension to higher dimensions}
\label{ss:fsa-dim}

We now complete the proof of Theorem \ref{thm:FSA} by treating the case of dimension bigger than 1.

\begin{prop}\label{prop:Rd}
Let $(\collS,\beta)$ be a fractional partition of $[M]$. Let $A_1,\ldots,A_M$ be nonempty compact subsets of $\R^d$. Then 
$$\big|A_1+\cdots+A_M\big|\ge \sum_{S\in \collS} \beta_S \, \Big|\sum_{i\in S} A_i\Big|,
$$
where $|A|$ denotes the $d$-dimensional Lebesgue measure of a compact subset $A$ of $\RL^d$.
\end{prop}

As before, it is enough to deal with the regular case. On each compact set $A_i$ the first
coordinate function $\pi$ (defined for $x\in \R^d$ by $\pi(x)=x_1$) achieves its maximum  at a point $a_i\in A_i$. Since our problem
is invariant by translation, we may assume without loss of generality that the minimum of $x\mapsto x_1$ on $A_i$ is achieved at the origin. 
So $\{0,a_i\}\subset A_i$ and $A_i\subset\{x\in \R^ d; \; x_1\in [0, (a_i)_1]\}$, where $(a_i)_1$ is the first coordinate of $a_i$.  In other words
$$ A_i\subset  \pi^{-1}\big( [0,\pi(a_i)]\big),$$
and their boundaries meet at least at $0$ and $a_i$.
The statement of Lemma \ref{lem:decomp} should be modified by 
replacing the intervals
$$ \left(\sum_{h_k^{-1}([1,j-1])}a\; ;\sum_{h_k^{-1}([1,j])}a\right]$$
by the slabs
$$ \pi^{-1}\left(\left(\sum_{h_k^{-1}([1,j-1])}\pi(a)\; ;\sum_{h_k^{-1}([1,j])}\pi(a)\right] \right).$$
The rest of the proof is the same  as in the one-dimensional case.

\section{Other proofs}
\label{sec:other-pf}

\subsection{Cartesian products}
\label{ss:cartes}

The next simple proposition allows combining fractional superadditivity results for volumes, and is the key tool in the proof of Theorem~\ref{th:prod}.

\begin{prop}\label{prop:cartes}
Let $d_1,d_2$ be positive integers and $p,q>0$. 
For $i\in[M]$, let $A_i\subset \R^{d_1}$ and $B_i\subset \R^{d_2}$ be nonempty compact sets. Let $(\beta_S)_{S\subset [M]}$ be non-negative numbers and assume 
that we have the following two volume inequalities:
\[
\Big| \sum_{i\in [M]}A_i\Big|_{d_1}^{\frac{1}{p}} \ge \sum_S  \beta_S\Big|\sum_{i\in S}A_i\Big|_{d_1}^{\frac{1}{p}} 
\quad\mbox{ and }\quad
\Big| \sum_{i\in [M]}B_i\Big|_{d_2}^{\frac{1}{q}} \ge \sum_S  \beta_S\Big|\sum_{i\in S}B_i\Big|_{d_2}^{\frac{1}{q}} .
\]
Then the Cartesian product sets $A_i\times B_i\in \R^{d_1+d_2}$ satisfy
\[
\Big| \sum_{i\in [M]}(A_i\times B_i)\Big|_{d_1+d_2}^{\frac{1}{p+q}} \ge \sum_S  \beta_S\Big|\sum_{i\in S}(A_i\times B_i)\Big|_{d_1+d_2}^{\frac{1}{p+q}}.
\]
\end{prop}
\begin{proof}
 Observe that $\sum_{i\in S}(A_i\times B_i)= \Big(\sum_{i\in S}A_i\Big)\times \Big(\sum_{i\in S} B_i\Big)$. Thus by H\"older's inequality
 \begin{eqnarray*}
 \sum_S  \beta_S\Big|\sum_{i\in S}(A_i\times B_i)\Big|^{\frac{1}{p+q}}
&=&  \sum_S  \beta_S\Big|\sum_{i\in S}A_i\Big|^{\frac{1}{p+q}}\Big|\sum_{i\in S}B_i\Big|^{\frac{1}{p+q}}\\
&\le & \left(  \sum_S  \beta_S\Big|\sum_{i\in S}A_i\Big|^{\frac{1}{p}}\right)^{\frac{p}{p+q}} \left(  \sum_S  \beta_S\Big|\sum_{i\in S}B_i\Big|^{\frac{1}{q}}\right)^{\frac{q}{p+q}}\\
&\le & \left(  \Big|\sum_{i\in [M]}A_i\Big|^{\frac{1}{p}}\right)^{\frac{p}{p+q}} \left(  \Big|\sum_{i\in [M]}B_i\Big|^{\frac{1}{q}}\right)^{\frac{q}{p+q}}\\
&=& \Big| \sum_{i\in [M]}(A_i\times B_i)\Big|^{\frac{1}{p+q}} .
 \end{eqnarray*}
 \end{proof}

\subsection{Proof for the integers}
\label{ss:integer-pf}

We now prove Theorem~\ref{th:Z} for cardinalities of sumsets in the integers.

The argument is the same as for the real line, but
with minor changes. Again we translate the sets in order to have $\min A_i=0$ and set $a_i:=\max A_i\in \mathbb Z$. 
Then we observe that the set on the left-hand side of  \eqref{eq:U1}  is included in 
$$ \left(0\; ;\sum_{h_k^{-1}([1,s])}a\right]=\left(0\; ;\sum_{[M]}a\right]$$
so it does not contain $0$, the minimal element of $\sum_{[M]}A$. So we may improve on  \eqref{eq:U1}:

\begin{equation}\label{eq:U1Z}
  \bigcup_{j=1}^s \left(\left(\sum_{S_j}A+\sum_{h_k^{-1}([1,j-1])\setminus S_j }a \right)\cap \left(\sum_{h_k^{-1}([1,j-1])}a\; ;\sum_{h_k^{-1}([1,j])}a\right] \right)\subset \Big( \sum_{[M]} A\Big) \setminus \{0\}.
  \end{equation}
Taking cardinalities gives  
  \begin{equation}\label{eq:U2Z}
   \sum_{j=1}^s \#\left(\bigg\{\sum_{S_j}A+\sum_{h_k^{-1}([1,j-1])\setminus S_j }a \bigg\}\cap \bigg(\sum_{h_k^{-1}([1,j-1])}a\; ;\sum_{h_k^{-1}([1,j])}a\bigg]\right)\le  \#\bigg( \sum_{[M]} A\bigg)-1.
 \end{equation} 
 Then we follow the same line of reasoning and get instead of \eqref{eq:recollect}:
 \begin{equation}
\label{eq:recollectZ}
q\left[\#\Big(\sum_{[M]}A\Big)-1\right] \ge 
\sum_{j=1}^s  \#\bigg(\sum_{S_j}A \cap \Big(0\; ;\sum_{S_j }a\Big] \bigg)=
 \sum_{j=1}^s \Big[ \#\bigg(\sum_{S_j}A \bigg) -1\Big],
\end{equation}
since $\min\left( \sum_{S_j}A\right)= 0$ and $\max\left( \sum_{S_j}A\right)=\sum_{S_j} a $. This concludes the proof in the regular case.
The general case follows.

\medskip
Eventually we check the sufficient condition for equality. Without loss of generality, we assume that for all $S$, $\beta_S>0$  and that $\min(A_i)=0$, $\max(A_i)=a_i$. Summing up the inclusions $\{0,a_i\}\subset A_i \subset [0,a_i]$ gives for all $S\subset [M]$,
\begin{equation}\label{eq:2inclusions} \left\{0,\sum_S a\right\}\subset \sum_S A \subset \left[0,\sum_S a\right]. \end{equation}
Since by hypothesis $\sum_S A$ is an arithmetic progression of increment $\rho$, we get 
$ \#(\sum_S A)= 1+\frac1\rho\sum_{i\in S} a_i$. Hence, using the fractional partition,
\[\sum_{S\in \mathcal G} \beta_S\left[
\#\Big(\sum_S A\Big) -1\right]=\sum_{S\in \mathcal G} \beta_S\left(\frac1\rho\sum_{i\in S} a_i\right) =\frac1\rho \sum_{i=1}^M a_i \left( \sum_{S\in\mathcal G; i\in S}  \beta_S\right)= \frac1\rho \sum_{i=1}^M a_i.\]
Each $i\in [M]$ belongs to some $S\in \mathcal G$, hence $A_i\subset \sum_S A \subset \rho \Z$, where we used that $0\in \cap_j A_j$ and that $\sum_S A$ is an arithmetic progression of increment $\rho$. Hence $\sum_{[M]} A \subset \rho \Z$, which together with 
\eqref{eq:2inclusions} implies that 
\[\#\Big(\sum_{[M]} A\Big)-1\le \frac1\rho\sum_{i=1}^M a_i=\sum_{S\in \mathcal G}\beta_S \left[\#\Big(\sum_S A\Big)-1\right].\]

\section{Concluding remarks and open questions}
\label{sec:rmks}

%
%
%

We leave a number of interesting open questions for future work.

\begin{itemize}

\item 
The question of characterizing all equality cases for our main inequality \eqref{eq:BMF1} in dimension 1 is interesting, and seems doable but tedious.

\item In a forthcoming paper \cite{BM24}, we use Theorem~\ref{thm:FSA} to show a certain monotonicity property in a limit theorem involving certain convolution powers of nonnegative measurable functions on the real line.

\item The central problem posed in this paper-- that of a full characterization of the Lyusternik region for $M>2$ -- seems quite difficult in general. It should however be possible to improve on our (inclusion) bounds or to put forward qualitative properties of these sets.  From the discussion of the counterexample
showing that partial sums cannot reach all supermodular set functions, it is clear
that characterizing the region would require at least to be able to say that if the two-by-two sums are not too big, then 
the sets are not far from convex and thus the three-by-three sum is not too big either. 
Such considerations lead towards refined stability results (see, e.g., \cite{FJ15}) and additive combinatorics,
and would be very interesting to pursue.

\item It is natural to ask what the analogue of the Lyusternik region
looks like when, instead of allowing all compact sets, one restricts to convex sets.
In this case, the question becomes clearly related to mixed volumes and their properties-- indeed, supermodularity properties of mixed volumes are discussed in \cite{FMZ22},
some properties of the reverse kind (log-submodularity) that hold for special subclasses of convex sets 
are discussed in \cite{FMZ22, FMMZ22}, and the possibility of extensions to more general measures absolutely continuous with respect to Lebesgue measure is discussed in \cite{FLMZ22}. 
Clearly the well known Alexandrov-Fenchel inequalities (see, e.g.,
\cite[Section 20.3]{BZ88:book} for a classical account and \cite{SH19,CKMS19} for recent developments) are also key constraints
on the collection of mixed volumes.
We remark that studies of regions involving the set of possible mixed volumes of convex bodies have been undertaken 
in a series of works in convex geometry (see, e.g., \cite{She60, HHS12, AS21}); however there does not appear to be a direct connection between
our work and those results because our interest is focused on what can be said for general compact sets.

\item Theorem~\ref{thm:beyondFSA} includes the observation from \cite{FMMZ18} that  $|A+B+C| + |A|$ may be strictly less than $|A+B| + |A+C|$ for compact sets $A, B, C$ even in dimension 1.
Nonetheless, \cite{FMMZ18} also show that if $A, B, C$ are compact subsets of $\R$, then
$$
|A+B+C| + |\conv(A)| \geq |A+B| + |A+C|.
$$
In particular, a supermodularity-type inequality holds if the set $A$ is convex (i.e., a closed interval). This may also be written as follows: if $A$ is a compact convex set and $B, C$ are arbitrary compact sets, and we define $\Delta_B(A)=|A+B|-|A|$, then 
$$
\Delta_{B+C}(A) \geq \Delta_B(A) + \Delta_C(A).
$$
This inequality was recently verified in general dimension when $B$ is a zonoid (and $C$ is an arbitrary compact set) by \cite{FMZ22}, 
but the question is open in general. 

\end{itemize}

\appendix

\section{The relevance of fractional superadditivity}
\label{sec:app}

In this Appendix, we discuss some motivations for considering fractional superadditivity a structural property of importance for set functions.

Our first observation, which is elementary but seemingly new, is that fractional superadditivity is closely connected to the extendability of a set function to a function on
the positive orthant with nice properties.
As usual we identify the set of subsets of $[M]$ with $\{0,1\}^M$ or with the set of applications from $[M]$ to $\{0,1\}$. 
In particular for $S\subset [M]$, the indicator function $\mathbf 1_S$ is viewed as a vector in  $\{0,1\}^M\subset \R^M$. The following result may be compared with the Lov\'asz extension theorem for submodular functions (see, e.g., \cite{Lov83, Fuj05:book}).

\begin{prop}\label{prop:extend}
Let $f:\{0,1\}^M\to \R_+$. Then the following assertions are equivalent:
\begin{enumerate}
	\item $f$ is fractionally superadditive, meaning that if $T\subset [M]$ and non-negative numbers $(\beta_S)_{S\subset [M]}$ satisfy  $\mathbf 1_T=\sum_{S\subset [M]} \beta_S \mathbf 1_S$ then
	$$f(\mathbf 1_T)\ge \sum_S \beta_S f(\mathbf 1_S).$$
	\item $f$ admits a 1-homogeneous concave extension to $\R^M_+=[0,+\infty)^M$.
\end{enumerate}
\end{prop}

\begin{proof}
$2\Longrightarrow 1$: Let $F$ be such an extension then for $\beta_S\ge 0$, set
$\beta:=\sum_S \beta_S$. Assume $\beta>0$ (otherwise all $\beta_S=0$ and the conclusion will be trivial). Then by homogeneity and concavity
$$ F\Big(\sum_S \beta_S \mathbf 1_S\Big)=\beta F\Big(\sum_S \frac{\beta_S}{\beta} \mathbf 1_S\Big) \ge \sum_S \beta_S F(\mathbf 1_S)= \sum_S \beta_S f(\mathbf 1_S).$$
So if $\mathbf 1_T=\sum_{S} \beta_S \mathbf 1_S$, we obtain
$f(\mathbf 1_T)=F(\mathbf 1_T)\ge \sum_S \beta_S f(\mathbf 1_S)$.

\medskip
$1\Longrightarrow 2$: 
 For $x\in \R^M_+$ we define
$$ F(x):=\sup\left\{ \sum_{S\subset M} \beta_S f(\mathbf 1_S)\; \Big| \; \beta_S\ge 0 \mbox{ s.t. } x=\sum_{S\subset M} \beta_S \mathbf 1_S\right\}.$$
Observe that the superadditivity condition, when applied to empty sets gives that $f(0)=f(\mathbf 1_{\emptyset})=0$. So we can also restrict the summation to $S\neq \emptyset$  in the supremum without changing its value.
Then $F(0)=0$ (as only $\beta_{\emptyset}$ may be nonzero and $f(1_\emptyset)=0)$. 

Let us consider $x\neq 0$ now. The above
set is not empty as $x=\sum_{i\in [M]} x_i \mathbf 1_{\{i\}}$ and the relationship
$x_i=\sum_{S;\; i\in S} \beta_S$ implies that for $S\neq \emptyset$, $\beta_S\in \big[0, \|x\|_{\infty}\big]$. So $F(x)$ is a well defined non-negative real number.
 One readily checks that $F$ is 1-homogeneous and concave. We have already seen 
 that $f(0)=F(0)=0$. By definition $F(\mathbf 1_T)\ge f(\mathbf 1_T)$ by choosing 
 the trivial decomposition of $1_S$ as itself. But for general decompositions
 $\mathbf 1_T=\sum_{S} \beta_S \mathbf 1_S$, fractional superadditivity
 gives that $\sum_S \beta_S f(\mathbf 1_S)\le f(\mathbf 1_T)$, so by taking supremum $F(\mathbf 1_T)\le f(\mathbf 1_T)$. 
 Consequently $F(\mathbf 1_T)= f(\mathbf 1_T)$ for every $T\subset [M]$.
\end{proof} 

It is tempting to try to find a simpler proof of Theorem~\ref{thm:FSA} by constructing a 1-homogeneous concave function that extends the set function $f(S)=|\sum_{i\in S} A_i|$. However, we have been unable to do this. We note that the obvious choice to consider is $F(x)=|\sum_{i\in [M]} x_i A_i |$, and  moreover, the concavity of this function is easy to check when each $A_i$ is a convex set using the Brunn-Minkowski inequality and the ``distributive'' property $(s+t)A = sA+ tA$ (which holds for $s, t>0$ if and only if $A$ is convex). However, the same idea to prove concavity of $F$ does not work for general compact sets because of the failure of the distributive property.

Our second observation, which is classical, is that fractional superadditivity (or ``balancedness'' as it is called in the economics literature) is equivalent to a certain ``nonempty core'' property of an optimization problem connected to the set function. This equivalence, proved by the duality theorem of linear programming, is the content of the Bondareva-Shapley theorem \cite{Bon63, Sha67} in the theory of cooperative games. We now state this theorem in our language and avoiding game-theoretic terminology.

Let $f:\{0,1\}^M\to \R_+$ with $f(\emptyset)=0$. Define the polyhedron
$$
A(f)=\bigg\{t\in \R_+^M: \sum_{i\in S} t_i \geq f(S) \text{ for each } S\subset [M]\bigg\}.
$$
The Bondareva-Shapley theorem states that $f$ is fractionally superadditive if and only if there exists $t\in A(f)$ such that $\sum_{i\in [M]} t_i=f([M])$.

The reader may consult \cite{Mad08:game} for a review of the cooperative game theory literature, including the Bondareva-Shapley theorem, from the viewpoint of information theory.

\par\vspace{.1in}
\noindent {\bf Acknowledgments:} 
M.M. was supported in part by the U.S. National Science Foundation (NSF) through grant DMS-1409504. This research was begun during the stay of the authors at the Isaac Newton Institute for the Mathematical Sciences, Cambridge, UK 
during the ``Discrete Analysis'' program in 2011; we are grateful to Liyao Wang for some useful discussions at that time.
Its completion was supported by the NSF under Grant No. 1440140, while the authors were in residence at the Mathematical Sciences Research Institute 
in Berkeley, California, for the ``Geometric and Functional Analysis and Applications'' program during the fall semester of 2017. We thank two anonymous referees for  detailed and pertinent feedback on an earlier version, which improved the clarity of the exposition in several places.

\par\vspace{.1in}
\noindent {\bf Data Availability:} 
Data sharing not applicable to this article as no datasets were generated or analysed during the current study.

\bibliographystyle{plain}

\end{document}